\documentclass[12pt,reqno]{amsart}
\usepackage{fullpage}
\usepackage{amsmath,amssymb,amsfonts,amsthm}

\usepackage{amsthm}
\usepackage{tikz}
\usepackage{enumitem}
\usepackage{graphicx}
\usepackage{pgfplots}
\usetikzlibrary{arrows}
\usepackage[colorinlistoftodos]{todonotes}
\usepackage[colorlinks=true, allcolors=blue]{hyperref}
\DeclareMathOperator{\dist}{dist}
\newtheorem{theorem}{Theorem}

\newtheorem{lemma}[theorem]{Lemma}

\newtheorem{remark}{Remark}
\pgfplotsset{compat=1.18} 
\begin{document}
\title[A One-Dimensional Planar Besicovitch-Type set]{A One-Dimensional Planar Besicovitch-Type set }
\author{Iqra Altaf}
\address{Department of Mathematics, The University of Chicago, 5734 S. University Avenue, Chicago, IL 60637, USA. email:
iqra@uchicago.edu.}
\keywords{}
\date{\today} 
\begin{abstract} A $\Gamma$-Besicovitch set is a set which contains a rotated copy of $\Gamma$ in every direction. Our main result is the construction of a non-trivial $1$-rectifiable set $\Gamma$ in the plane, for which there exists a 1-dimensional $\Gamma$-Besicovitch set. 
\end{abstract} 
\maketitle 
\tableofcontents
\section{Introduction}
The null sets in $\mathbb{R}^{n}$ which contain a unit line segment in every direction are called Besicovitch sets or Kakeya sets. In the plane it is well-known that every Besicovitch set has Hausdorff dimension $2$. For $n > 2$, it is conjectured that every Besicovitch set in $\mathbb{R}^{n}$ must have full Hausdorff dimension.\par 
It is an interesting question to ask what happens if the line segment is replaced by other rectifiable sets. In \cite{CC}, it is shown that for any rectifiable planar set $\Gamma$ there exists a set of measure zero which contains a rotated copy of a full $\mathcal{H}^1$-measure subset of $\Gamma$ in each direction. This raises the question what the Hausdorff dimension of such sets can be. \par 
A $\Gamma$-Besicovitch set is a set which contains a rotated copy of $\Gamma$ in every direction.
There are some trivial examples of rectifiable sets for which there exist $\Gamma$-Besicovitch sets of Hausdorff dimension less than 2. For example if $\Gamma$ is a circular arc then the complete circle containing $\Gamma$ is a $\Gamma$-Besicovitch set. If a rectifiable set $\Gamma$ is covered by a concentric union of circles $\mathcal{C}$ then $\mathcal{C}$ is also a $\Gamma$-Besicovitch set. Such a set $ \mathcal{C}$ can have any dimension between 1 and 2. More interestingly, it has also been mentioned in \cite{CC} that if $\Gamma$ is an arbitrary countable union of circles, then there exists a $\Gamma$-Besicovitch set of Hausdorff dimension $1$. The circles do not have to be concentric. We also know that if $\Gamma$ is a $C^3$ curve and not a circular arc then the $\Gamma$-Besicovitch set has dimension $2$, see \cite{PYZ} and \cite{Z}. In both of these papers, the curvature of $\Gamma$ plays a central role in the proof of this statement.\par
 In \cite{IMK}, the authors consider Besicovitch sets of Cantor graphs built on a special class of Cantor sets. These Cantor sets do not have to be self-similar but have a certain degree of symmetry. (For details, see \cite{IMK}). The Cantor graphs are rectifiable sets that resembles the devil's staircase graph. The authors show that for such a Cantor graph $\Gamma$ built on a Cantor set $C$, every $\Gamma$-Besicovitch set has Hausdorff dimension greater than $\min\left(2-s^2,\frac{1}{s}\right)$, where $s=\dim C$.  \par 
 It was a well known open problem asked by Marianna Cs\"ornyei whether there exists a non-trivial $\Gamma$-Besicovitch set of dimension less than 2. Until now, the only known $\Gamma$-Besicovitch sets of dimension less than 2 were sets covered by lower dimensional union of circles. The construction of measure zero sets by Cs\"ornyei and Chang in \cite{CC} depended upon the tangent field of $\Gamma$. It is well-known that every $1$-rectifiable set in the plane has a tangent field which is defined almost everywhere. For a circular arc all the tangent lines have different directions and for a straight line, all the tangent lines are parallel. This suggested that it is interesting to look at the $\Gamma$-Besicovitch sets for rectifiable sets with the same tangent field as that of a line i.e. rectifiable sets whose tangents are parallel almost everywhere. \par
  Marianna Cs\"ornyei asked, in particular, does there exist a $1$-rectifiable set with parallel tangents almost everywhere for which there exists a $\Gamma$-Besicovitch set of dimension less than 2?  Apart from being an interesting question in itself, it may help us in understanding the (usual) Besicovitch set. In this paper we are answering this question affirmatively. Our main theorem is the following.
\begin{theorem} \label{main} There exists a $1$-rectifiable set $\Gamma$ with the following properties. 
\begin{enumerate} [label=(\roman*)]
\item The 1-dimensional Hausdorff measure of $\Gamma$ is positive.
    \item There exists a 1-dimensional $\Gamma$-Besicovitch set.
    
    \item The set $\Gamma$ is the graph of a monotone function whose domain is a Cantor set of measure $0$ and thus the tangents of $\Gamma$ are vertical almost everywhere.
\end{enumerate}
For any $s \in [0,1]$, we can construct the set $\Gamma$ such that the Cantor set in (iii) has Hausdorff dimension $s$.
\end{theorem}
\begin{remark}  \label{remareas}As the set $\Gamma$ has positive length and parallel tangents almost everywhere, it is straightforward to see that it cannot be covered by a union of concentric circles whose dimension is less than 2. 
\end{remark} 
\noindent \textbf{Notation} For $ d \in \{1,\,2\}$ and a set $E \subset \mathbb{R}^{d}, \, |E|$ denotes the $d$-dimensional Lebesgue measure of $E$. For $z \in \mathbb{R}^2$, $P_{x}(z)$ and $P_{y}(z)$ denote the $x$ and the $y$ coordinates of $z$ respectively, and $\|z\|$ denotes the Euclidean norm of $z$. For simplicity of notation we identify the complex plane $\mathbb{C}$ with $\mathbb{R}^2$, i.e. $(x,y) \cong x+iy$. For every $a, b \in \mathbb{R}$, $a \lesssim b$ implies that there is an absolute constant $C$ such that $a \leq Cb$, and $a \simeq b$ implies that $a \lesssim b$ and $b \lesssim a$. \\ Throughout the paper, dimension refers to Hausdorff dimension. 
\section{Construction of $\Gamma$} \label{cavt}
We will construct the 1-rectifiable set $\Gamma$ mentioned in Theorem \ref{main} as the countable intersection of the sets $\{\mathcal{S}_n\}_{n \in \mathbb{N}}$. The sets $\mathcal{S}_n$ will be constructed inductively such that  they are nested, i.e. $\mathcal{S}_{n+1} \subset \mathcal{S}_n$. For each $n \in \mathbb{N}$,  $\mathcal{S}_n$ will be a union of finitely many rectangles with sides parallel to the coordinate axes. We will introduce two positive sequences $\{\delta_{n}\}_{n \in \mathbb{N}}$ and $\{\Delta_{n}\}_{n \in \mathbb{N}}$ such that for each rectangle in $\mathcal{S}_n$, its $x$-projection has length $\delta_n$ and its $y$-projection has length close to $\Delta_{n}$. \par  

Let $c<1$ be a small constant to be specified later. Let $\{\delta_{n}\}_{n \in \mathbb{N}}$ and $\{\Delta_{n}\}_{n \in \mathbb{N}}$ be positive decreasing sequences converging to $0$ with the following properties.
\begin{align}
& \delta_1=\Delta_1=1. \hspace*{ 12 cm} \label{P1} \\
&\Delta_{n+1} \leq c\delta_{n}^{n+1}. \label{P2}\\
&\delta_{n+1} \leq c\Delta_{n+1} \delta_{n}. \label{P3} \\ 
& \Delta_2^{-1},  \hspace*{.0612 cm} 
\Delta_{n+1}^{-1} 
 \delta_{n}^{-1}\Delta_{n}\delta_{n-1} \in \mathbb{N}, \text{ for all }n \geq 2.   \label{P4}
\end{align} 
We see from \eqref{P3} that $\delta_{n}\Delta_{n}^{-1}\rightarrow 0$. This means the rectangles in $\mathcal{S}_n$ get narrower, in the sense that the ratio of the length of the side parallel to the $x$-axis to the length of the side parallel to the $y$-axis will go to $0$ as $n \rightarrow \infty$. \par 
We define $\theta_1:=c$ and $\theta_{n+1}:=c\Delta_{n+1}\delta_{n}$, for all $n \geq 1$. The property \eqref{P4} above implies that 
\begin{equation} \label{pe}
    \theta_{n+1}^{-1} \theta_{n} \in \mathbb{N}, \text{ for all }n \geq 1.
\end{equation}
\noindent For our construction of $\Gamma$, we will later need to choose points $\alpha_{n+1}$ and $q_{n+1}$ on a circular arc $C_{n}$ such that $\alpha_{n+1}$ is below the $x$-axis and is the center of $C_n$, and the following hold. (See Figure 1)
\begin{enumerate} [label=(\roman*)]
    \item The arc $C_n$ lies in the first quadrant of $\mathbb{R}^2$ and joins the points (0,0) and $(\delta_n,\Delta_n)$.
    \item The point $q_{n+1}$ is the intersection of $C_n$ with the line $y=\Delta_{n+1}$.
    \item The sub-arc of $C_n$ joining (0,0) and $q_{n+1}$ has angle $\theta_{n+1} = c\Delta_{n+1}\delta_{n}$.
\end{enumerate}
\hspace* {3cm} \\
   \begin{figure}[ht]
    \centering
    \includegraphics[width=\linewidth]{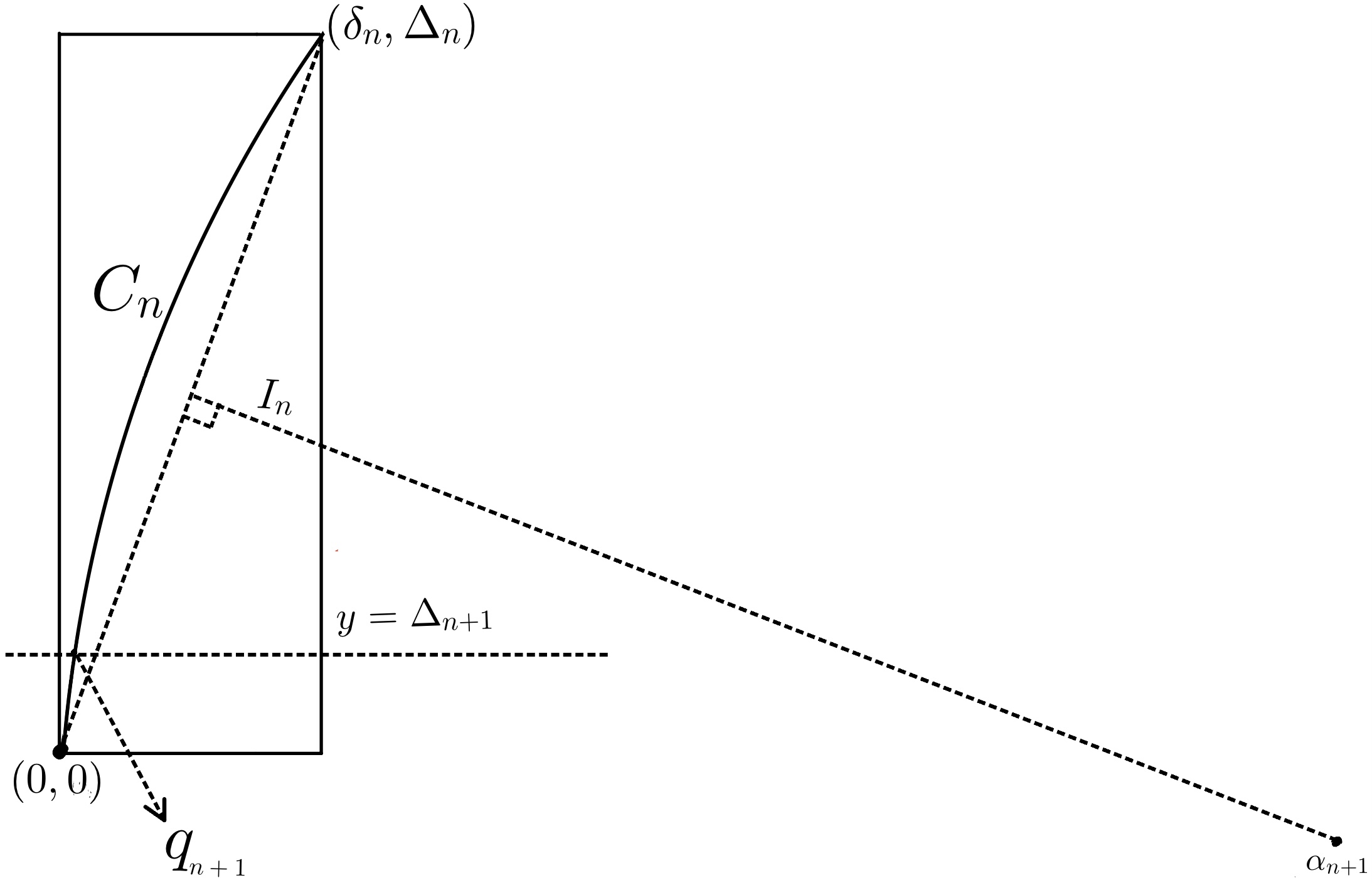}
    \caption{ }
    \label{figu}
\end{figure}
To see why we can choose such $C_n$, $\alpha_{n+1}$ and $q
_{n+1}$, let $\mathcal{C}$ be any circular arc in the first quadrant, joining (0,0) and $(\delta_n,\Delta_n)$. Let $q$ be the intersection of $\mathcal{C}$ with the line $y=\Delta_{n+1}$ and let $\theta$ be the angle of the arc joining $(0,0)$ and $q$. For any circular arc to pass through the points $(0,0)$ and $(\delta_n,\Delta_{n})$, its center must lie on the perpendicular bisector of the line segment joining these two points, call it $I_{n}$. Let $\alpha \in I_{n}$ be the center of the circular arc $\mathcal{C}$.  If $\lVert \alpha \rVert$  $\rightarrow \infty$, the arc $\mathcal{C}$ converges to the line segment joining $(0,0)$ and $(\delta_{n}, \Delta_{n})$ and the angle of $\mathcal{C}$ converges to 0. 
An elementary calculation shows that if $\alpha $ lies on the $x$-axis, then $\alpha=(\frac{\delta_n^2+\Delta_n^2}{2\delta_n},0)$. Since the radius of $\mathcal{C}$ is $\lVert \alpha \rVert $, the length of the arc joining (0,0) and $q$ is $\lVert \alpha \rVert \theta$. Using $\delta_{n} \leq \Delta_{n}$ we have
    \begin{align*}
    \Delta_{n+1} \leq \lVert q \rVert & \leq \lVert \alpha \rVert \theta =\Big( \frac{\delta_{n}^2+\Delta_{n}^2}{2\delta_{n}}\Big)\theta \leq \frac{\Delta_{n}^2}{\delta_{n}}\theta.
    \end{align*}
 Thus in this case $\theta \geq \Delta_{n+1}\Delta_n^{-2}\delta_{n}.$   \par 
 Since  $\Delta_{n} \leq 1$ and $c<1$ we have $0<c\Delta_{n+1}\delta_{n} \leq \Delta_{n+1}\Delta_n^{-2}\delta_{n}$. So we can choose $\alpha_{n+1} \in I_{n}$ below the $x$ axis such that $\theta= \theta_{n+1}=c\Delta_{n+1}\delta_{n}$. We denote the corresponding $\mathcal{C}$ and $q$ by $C_n$ and $q_{n+1}$. (See Figure 1). 

We are now ready to construct $\Gamma$. We start the inductive process by defining $\mathcal{S}_{1}:= Q_{1,1}:=[0,1]\times [0,1]$. Assume we have constructed $\mathcal{S}_n=\bigcup\limits_{j}Q_{n,j}$, where the union is over finitely many axis parallel rectangles $Q_{n,j}$, whose horizontal sides are disjoint and vertical sides are non-overlapping.
We denote the bottom-left corner of $Q_{n,j}$ as $p_{n,j}$. We also assume that $Q_{n,j}$ satisfies the following properties. \begin{itemize} 
\item[] \leavevmode\vspace*{-\dimexpr\abovedisplayskip + \baselineskip}\begin{align}\label{Bla}
    P_{x}(p_{n,j}) < P_{x}(p_{n,j+1}) \text{ and } P_{y}(p_{n,j}) < P_{y}(p_{n,j+1})
\end{align} 
\item[] \leavevmode\vspace*{-\dimexpr\abovedisplayskip + \baselineskip}\begin{align} \label{A}
Q_{n,1}=[0,\delta_n]\times [0,\Delta_n].
\end{align} 
\item[] \leavevmode\vspace*{-\dimexpr\abovedisplayskip + \baselineskip}\begin{align} \label{B}
|P_{x}(Q_{n,j})|=\delta_{n}.
\end{align}
\end{itemize}
    \noindent For each $j$, we will construct finitely many rectangles $Q_{n,j,k}$ contained in $Q_{n,j}$. We first define the bottom-left corners of the rectangles $Q_{n,j, k}$ which we call $p_{n,j, k}$. For $k \leq 2\pi \theta_{n+1}^{-1}$, let the point $p_{n,j,k}$  be
  \begin{equation} \label{aliph}
     p_{n,j,k}=\alpha_{n+1}(1-e^{-i (k-1)\theta_{n+1}})+ e^{-i (k-1)\theta_{n+1}}p_{n,j}.  
\end{equation} 
Putting $k=1$ in the above equation, we see $p_{n,j,1}=p_{n,j}$. From \eqref{A}, we see that $p_{n,1}=(0,0).$ This means that for $j=1$ we have 
\begin{equation} \label{aliphu} p_{n,1,k}=\alpha_{n+1}(1-e^{-i (k-1)\theta_{n+1}}).\end{equation} 
Thus the point $p_{n,1,k}$ is the image of $(0,0)$ under rotation by $-(k-1)\theta_{n+1}$ around the point $\alpha_{n+1}$ (see Figure 2) and $p_{n,1,2}=q_{n+1}$. 
 From \eqref{aliph} we also get
\begin{align} \label{zo}
     p_{n,j,k+1}-p_{n,j,k}& =\alpha_{n+1}(e^{-i (k-1)\theta_{n+1}}-e^{-i k\theta_{n+1}})+ \big(e^{-i k\theta_{n+1}}-e^{-i (k-1)\theta_{n+1}}\big)p_{n,j} \nonumber \\
     & = e^{-i(k-1)\theta_{n+1}} (1-e^{-i\theta_{n+1}})( \alpha_{n+1}-p_{n,j}). 
\end{align}
Thus we have 
\begin{equation*} 
|p_{n,j,k+1}-p_{n,j,k}|=|(1-e^{-i\theta_{n+1}})( \alpha_{n+1}-p_{n,j})|.
\end{equation*} 
Therefore for any fixed $j$, the points $p_{n,j,k}$  are equidistant, as the modulus of their difference is independent of $k$.
 \begin{figure}[ht]
    \centering
    \includegraphics[width=\linewidth]{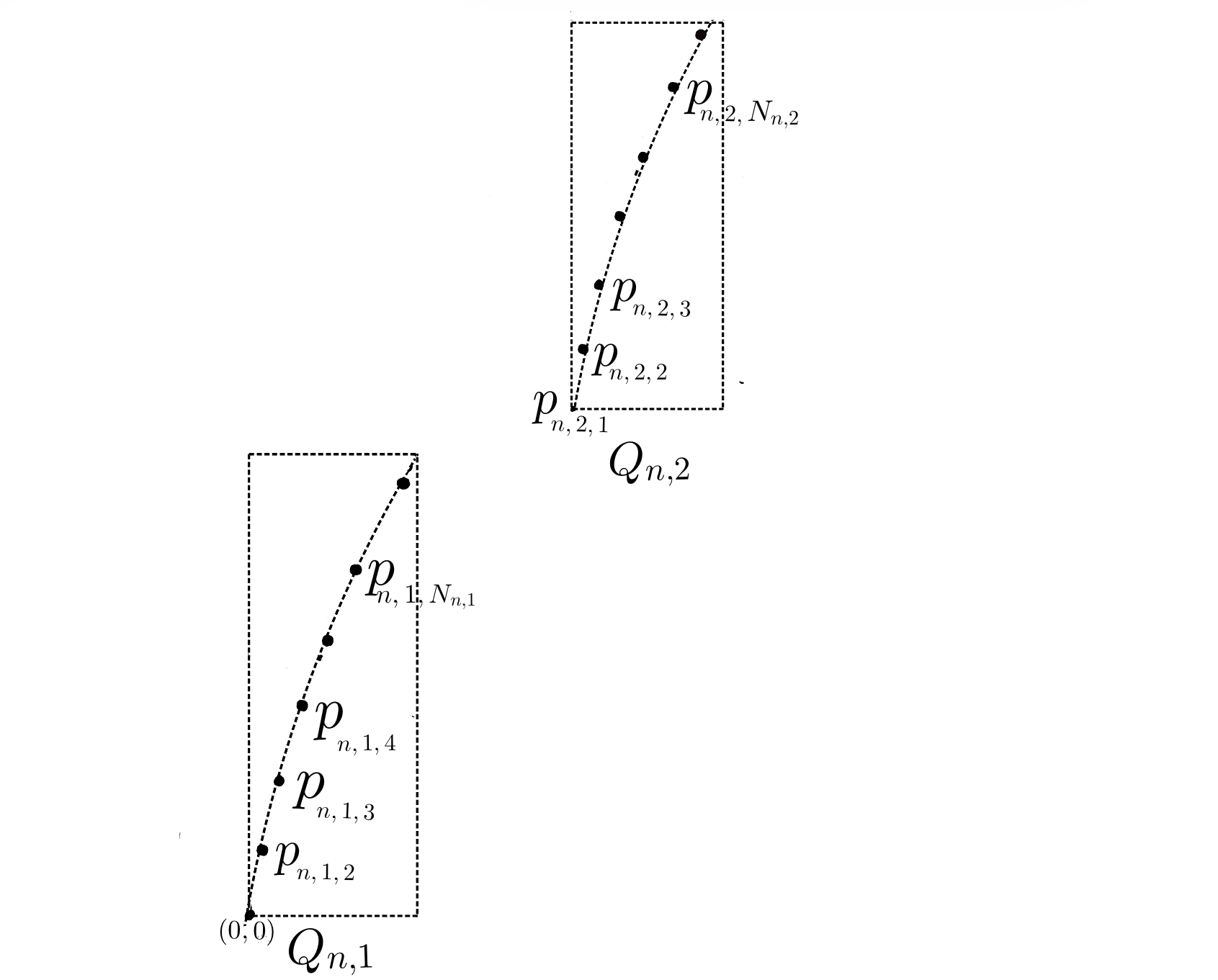}
    \caption{ }
    \label{figur}
\end{figure}
 We also have from \eqref{aliph}\begin{align} \label{lam}
 p_{n,j,k+l}-e^{-il  \theta_{n+1}}p_{n,j,k}= \alpha_{n+1}(1-e^{-i l\theta_{n+1}}) = p_{n,1,l+1}.\end{align}
 Equation \eqref{lam} implies that the point $p_{n,j,k+l}$ is the image of $p_{n,j,k}$ under rotation about the origin by an angle of $-l\theta_{n+1}$ and translation by $p_{n,1,l+1}$. We will use this fact to define the $\Gamma$-Besicovitch set later.
  \par
 If $p_{n,j,k+1} \in Q_{n,j}$ we define the rectangle $Q_{n,j,k}$  with sides parallel to the $x$ and the $y$ axes such that 
 \begin{equation} \label{ali}
     P_x(Q_{n,j,k})=[P_x(p_{n,j,k}), P_x(p_{n,j,k})+\delta_{n+1}]
\end{equation}  and 
\begin{align} \label{al}
    P_y(Q_{n,j,k})=[P_y(p_{n,j,k}), P_y(p_{n,j,k+1})].
\end{align}
 For two rectangles $Q_{n,j,k}$ and $Q_{n,j',k'}$ with $j\neq j'$, their $y$-projections are non-overlapping and their $x$-projections are disjoint, because that is the case for the larger rectangles $Q_{n,j}$ and $Q_{n,j'}$. However, we need to justify this property for rectangles $Q_{n,j,k}$ and $Q_{n,j,k'}$ which lie in the same larger rectangle $Q_{n,j}$. It will follow from Lemma \ref{le} that if $p_{n,j,k} \in Q_{n,j}$ then 
$$P_{y}(p_{n,j,k+1})>P_y(p_{n,j,k}) \text{ and }  P_{x}(p_{n,j,k+1})>P_x(p_{n,j,k})+ \delta_{n+1}.$$
Therefore the condition $p_{n,j,k+1} \in Q_{n,j}$, along with the above inequalities, implies that the rectangles $Q_{n,j,k}$ lie inside $Q_{n,j}$ and have disjoint $x$-projections and non-overlapping $y$-projections. \par 
\noindent We establish some technical properties of the points $p_{n,j,k}$ which lie inside the unit square.
 \begin{lemma} \label{le} Let $n \geq 1$ and $p_{n,j,k} \in [0,1]\times [0,1]$. Then for an absolute constant $c_1$, we have
    \begin{enumerate}[label=(\roman*)]
  \item $ |\Delta_{n+1}-\big(P_{y}(p_{n,j,k+1})-P_{y}(p_{n,j,k})\big)|\leq c_1\theta_{n+1}.$ 
  \item $|\delta_n \Delta_{n}^{-1}\Delta_{n+1}-\big(P_{x}(p_{n,j,k+1})-P_{x}(p_{n,j,k}) \big)| \leq 3c_1\theta_{n+1}.$
  \end{enumerate}
 \end{lemma}
\begin{proof} By \eqref{lam} we have $p_{n,1,2} =p_{n,j,k+1}-e^{-i\theta_{n+1}}p_{n,j,k}.$ Therefore
\begin{align} \label{bru} p_{n,1,2}-(p_{n,j,k+1}-p_{n,j,k})& =
 p_{n,j,k+1}-e^{-i\theta_{n+1}}p_{n,j,k}-(p_{n,j,k+1}-p_{n,j,k}) \nonumber \\ 
& = (1-e^{-i\theta_{n+1}})p_{n,j,k}.
\end{align} 
As $|1-e^{-i\theta_{n+1}}|\lesssim \theta_{n+1}$ and  $|p_{n,j,k}|\lesssim 1$, we have 
\begin{align} \label{panda}
    |p_{n,1,2}-(p_{n,j,k+1}-p_{n,j,k})|= |(1-e^{-i\theta_{n+1}})p_{n,j,k}|\leq c_1\theta_{n+1}.
\end{align}
Notice that $c_1$ is an absolute constant and it does not depend on $c$. Recall that $P_{y}(p_{n,1,2})= \Delta_{n+1}$. Therefore (i) holds. \par
Recall that the circular arc $C_n$ is in the first quadrant with center below the $x$-axis. So it is the graph of a concave function. Moreover, the line segment joining (0,0) and $(\delta_{n}, \Delta_{n})$ makes an angle greater than $\pi/4$ with the $x$-axis. Thus $C_n$ is also the graph of a monotonic function and 
$$0 < P_y(p_{n,1,k+1})-P_y(p_{n,1,k}) \leq P_{y}(p_{n,1,2}) =\Delta_{n+1}.$$ 
Thus the number of points $p_{n,1,k}$ on $C_n$ is at least $\Delta_{n}\Delta_{n+1}^{-1}-1$. By \eqref{P2} and \eqref{P3}, we have $\Delta_{n}\Delta_{n+1}^{-1} \geq c^{-1}\Delta_{n}\delta_{n}^{-(n+1)} \geq c^{-1}.$ We choose $c$ very small, in particular $c<1/10$. Thus there are at least 9 points $p_{n,1,k}$ on $C_n$. By concavity, the slope of the line joining the points $p_{n,1,k}$ and $p_{n,1,k+1}$ decreases as $k$ increases. The slope of the vector $p_{n,1,2}$ is greater than $\Delta_n \delta_{n}^{-1}$. If $p_{n,1,k+1}$ is the closest point to the corner $(\delta_n, \Delta_{n})$, then $ \dist(p_{n,1,k}, (\delta_{n}, \Delta_{n})) < \dist (p_{n,1,k}, (0,0))$. Thus the slope of $p_{n,1,k+1}-p_{n,1,k}$ will be less than $\Delta_{n}\delta_{n}^{-1}$. Therefore, using \eqref{panda} we have the following.
$$\frac{\Delta_{n+1}-c_1 \,\theta_{n+1}}{P_x(p_{n,1,2})+c_1\,\theta_{n+1}} = \frac{P_{y}(p_{n,1,2})-c_1\,\theta_{n+1}}{P_x(p_{n,1,2})+c_1\,\theta_{n+1}} \leq \frac{P_{y}(p_{n,1,k+1}-p_{n,1,k})}{P_{x}(p_{n,1,k+1}-p_{n,1,k})}\leq \frac{\Delta_{n}}{\delta_{n}}.$$ Thus,  
\begin{align} \label{di}
\delta_{n}\Delta_{n}^{-1}\Delta_{n+1}-P_x(p_{n,1,2}) \leq c_1\,\theta_{n+1}\delta_{n}\Delta_{n}^{-1}+c_1\,\theta_{n+1}\leq 2c_1\theta_{n+1}.
\end{align} Combining this with \eqref{panda} gives (ii).
\end{proof}
 \noindent It follows from Lemma \eqref{le} (i) that 
 $$P_{y}(p_{n,j,k+1})-P_y(p_{n,j,k})\geq \Delta_{n+1}-c_1\theta_{n+1}=\Delta_{n+1}-cc_1\Delta_{n+1}\delta_{n}=\Delta_{n+1}(1-cc_1\delta_{n}).$$ 
 From Lemma \ref{le} (ii), we also have 
 $$P_{x}(p_{n,j,k+1})-P_x(p_{n,j,k}) \geq \delta_{n}\Delta_{n}^{-1}\Delta_{n+1} -3c_1\theta_{n+1}> \delta_{n}\Delta_{n}^{-1}\Delta_{n+1}(1-3cc_1\Delta_{n}) .$$ 
 We choose $c$ to be small enough such that $1/2<1-3cc_1 <1-3cc_1\Delta_{n}$ and $c <1/10$. Therefore the $y$-projections of the rectangles $Q_{n,j,k}$ are non-overlapping. Using \eqref{P3} we have
 \begin{equation} \label{bux}
  P_{x}(p_{n,j,k+1})-P_x(p_{n,j,k}) \geq \delta_{n}\Delta_{n}^{-1}\Delta_{n+1}/2 \geq 3c\delta_{n}\Delta_{n}^{-1}\Delta_{n+1} \geq 3c\delta_{n}\Delta_{n+1} \geq 3\delta_{n+1}.   
 \end{equation}
 Recall that $|P_x(Q_{n,j,k})|=\delta_{n+1}$. Thus the $x$-projections of the $Q_{n,j,k}$ are disjoint. \par
 We define $N_{n,j}$ to be the number of rectangles $Q_{n,j,k}$ inside the bigger rectangle $Q_{n,j}$, i.e.
\begin{equation} \label{h}
    N_{n,j}={\#}\{k : Q_{n,j,k} \subset Q_{n,j} \} \text{ and } N_n=\min_{j} N_{n,j}.
\end{equation} 
  Define 
\begin{equation} \label{Izz}
    \mathcal{S}_{n+1}= \bigcup_{j} \bigcup_{k \leq N_n} Q_{n,j,k}.
\end{equation} We relabel the rectangles ${Q_{n,j,k}}$ as $Q_{n+1,i}$ in the order of distance from the origin. Notice that the rectangles $Q_{n+1,i}$ satisfy \eqref{Bla}, \eqref{A} and \eqref{B} with $n$ replaced by $n+1$. We continue the iterative process for all $n \in \mathbb{N}$ and define
\begin{equation} \label{pak}
  \Gamma= \bigcap_{n=1}^{\infty} \mathcal{S}_{n}.
 \end{equation}
 The next lemma gives an estimate on $N_{n}$ and on the $1$-dimensional Hausdorff measure of $\Gamma$.
 \begin{lemma} \label{ba} For all $n \geq 1$, we have the following.
 \begin{enumerate}[label=(\roman*)]
     \item $\Delta_{n}\Delta_{n+1}^{-1}(1-c_2\delta_{n-1})\leq N_{n} \leq \Delta_{n}\Delta_{n+1}^{-1}(1+c_2\delta_{n-1})$, where $c_2=4c(c_1+1)$.
     \item $\#\{j: Q_{n,j} \subset \mathcal{S}_{n}\} \gtrsim \Delta_{n}^{-1}$ and $ |P_{y}(\mathcal{S}_{n})| \gtrsim 1.$ 
 \end{enumerate} 
 \end{lemma}
  \begin{proof} By the construction of the rectangles, we have 
  $$\# \{k: Q_{n,j,k} \subset Q_{n,j}\}  \geq \min \displaystyle \Big\{\frac{|P_{y}(Q_{n,j})|}{\max_{k}\, (|P_{y}(Q_{n,j,k})|)}, \frac{|P_{x}(Q_{n,j})|}{\max_{k} \, (|P_{x}(p_{n,j,k+1}-p_{n,j,k})|)}\Big\}-1.$$
   From Lemma \ref{le} we have that $|P_{y}(Q_{n,j})-\Delta_{n}| \lesssim \theta_{n}$, for all $n$ and $j$. Using Lemma \ref{le} (i) and (ii), we have
 \begin{align*} 
     \# \{k : Q_{n,j,k} \subset Q_{n,j}\}  \geq \min \Big\{\frac{\Delta_{n}-c_1\,\theta_{n}}{\Delta_{n+1}+c_1\theta_{n+1}}, \frac{\delta_{n}}{\delta_{n}\Delta_{n}^{-1}\Delta_{n+1}+3c_1\,\theta_{n+1}}\Big\} -1.
 \end{align*}  Recall $\theta_{n}=c\Delta_{n}\delta_{n-1}$ and $\theta_{n+1}=c\Delta_{n+1}\delta_{n}$. Therefore,
 \begin{align*}
     \frac{\Delta_{n}-c_1\,\theta_{n}}{\Delta_{n+1}+c_1\theta_{n+1}}=\frac{\Delta_{n}\Delta_{n+1}^{-1}(1-cc_1\delta_{n-1})}{1+cc_1\delta_{n}}\, 
 \text{ and } 
    \,  \frac{\delta_{n}}{\delta_{n}\Delta_{n}^{-1}\Delta_{n+1}+3c_1\,\theta_{n+1}} = \frac{\Delta_{n}\Delta_{n+1}^{-1}}{1+3cc_1\Delta_n}.
 \end{align*}
 Now we use the inequality that $\frac{1}{1+x} > 1-x$ to obtain
 $$ \frac{(1-cc_1\delta_{n-1})} {1+cc_1\delta_{n}} \geq (1-cc_1\delta_{n-1})(1-cc_1\delta_{n}) \geq 1-2cc_1\delta_{n-1},$$ and 
 $$\frac{1}{1+3cc_1\Delta_n} \geq 1-3cc_1\Delta_{n}.$$
 Finally using \eqref{P2}  we have $\Delta_{n} \leq \delta_{n-1}$ and 
 \begin{align*}N_{n,j} \geq  \Delta_{n}\Delta_{n+1}^{-1}(1-3cc_1\delta_{n-1}) -1 & \geq \Delta_{n}\Delta_{n+1}^{-1}(1-3cc_1\delta_{n-1}-\Delta_{n}^{-1}\Delta_{n+1}) \\
 & \geq \Delta_{n}\Delta_{n+1}^{-1}(1-3cc_1\delta_{n-1}-c\Delta_{n}^{n}) \\
 & \geq \Delta_{n}\Delta_{n+1}^{-1}(1-3cc_1\delta_{n-1}-c\delta_{n-1}) \\
 & \geq \Delta_{n}\Delta_{n+1}^{-1}(1-4c(c_{1}+1)\delta_{n-1}).
 \end{align*}
 Since the above lower bound holds for all $j$, we have the left hand side of (i).
 Similarly, 
 \begin{align*}
     N_{n,j}  & \leq {\frac{|P_{y}(Q_{n,j})|}{\min_{k} (|P_{y}(Q_{n,j,k})|)}}   
     \leq \frac{\Delta_{n}+c_1\,\theta_{n}}{\Delta_{n+1}-c_1\theta_{n+1}} = \frac{\Delta_{n}\Delta_{n+1}^{-1}(1+cc_1\delta_{n-1})}{1-cc_1\delta_{n}}.
\end{align*}Now we use the inequality that $\frac{1}{1-x} < 1+2x$, for small $x$, to obtain 
$$ \frac{1+cc_1\delta_{n-1}}{1-cc_1\delta_{n}} \leq (1+cc_1\delta_{n-1})(1+2cc_1\delta_{n}) \leq 1+4cc_1\delta_{n-1}.$$ Therefore we have the right hand side of (i).\par 
By \eqref{Izz}, the number of rectangles obtained in $(n+1)^{st}$ step $\#\{(j,k) : Q_{n,j,k} \subset \mathcal{S}_{n+1}\}$ is $\#\{j : Q_{n,j} \subset \mathcal{S}_{n+1}\}N_{n}$. From (i), we have 
  \begin{equation*}
  \Delta_{n}\Delta_{n+1}^{-1}(1-c_2\delta_{n-1}) \#\{j: Q_{n,j} \subset \mathcal{S}_{n}\}   \leq \#\{(j,k) : Q_{n,j,k} \subset \mathcal{S}_{n+1}\}. \end{equation*} Rewriting this equation, 
  \begin{equation*}
   \#\{ j: Q_{n,j} \subset \mathcal{S}_{n}\}{\Delta_{n}}(1-c_2\delta_{n-1}) \leq \#\{(j,k): Q_{n,j,k} \subset \mathcal{S}_{n+1}\}{\Delta_{n+1}} .
  \end{equation*} Recall $\#\{Q_{1,1}\}{\Delta_{1}}=1$, thus we have 
  \begin{equation} \label{Ramen}
    \prod_{l=1}^{n-1}(1-c_2\delta_{l})  \leq  \#\{(j,k): Q_{n,j,k}\subset \mathcal{S}_{n+1}\}{\Delta_{n+1}} .  
  \end{equation}
The product $\prod_{l=1}^{n}(1-c_2\delta_{l}) $ converges to a positive number as $n \rightarrow \infty$. This is because $\sum_{l}\delta_{l} < \infty$.  Therefore 
$$\#\{j: Q_{n,j}\subset \mathcal{S}_{n}\} \gtrsim \Delta_{n}^{-1}, \text{ for all } n \in \mathbb{N}.$$
From Lemma \ref{le} (i) we have $ |P_y(Q_{n,j})| \simeq \Delta_{n}$. Thus 
  $$ |P_{y}(\mathcal{S}_{n})| =\sum_{j}|P_{y}(Q_{n,j})| \gtrsim \Delta_{n}^{-1}\Delta_{n} = 1.$$
  This completes the proof of (ii).
   \end{proof}
\noindent It follows from Lemma \ref{ba} (ii) that 
  \begin{equation}
      |P_{y}(\Gamma)| =\lim_{n \rightarrow \infty} |P_{y}(\mathcal{S}_{n})| \simeq 1. \end{equation} Thus $\Gamma$ has positive $1$-dimensional Hausdorff measure. Clearly $\Gamma$ is a subset of a graph of a monotonous continuous function and thus is 1-rectifiable. Moreover,
      \begin{equation}
      |P_{x}(\Gamma)| \leq \delta_n \#\{j: Q_{n,j}\subset \mathcal{S}_{n}\} \simeq \delta_n \Delta_{n}^{-1}.\end{equation} 
      Since $\delta_{n}\Delta_{n}^{-1} \rightarrow 0$ the domain is a Cantor set of measure 0 and therefore $
      \Gamma$ has vertical tangents almost everywhere. Thus we have shown property (i) and (iii) of Theorem \ref{main}. \par 
 We end this section with another estimate on $N_{n}$. 
 \begin{lemma}\label{Boow} For all $n \in \mathbb{N}$, we have $N_{n} < \theta_{n} \theta_{n+1}^{-1}$.  
     \end{lemma}
     \begin{proof} From Lemma \ref{ba} (i) we have 
     $ N_{n} \leq \Delta_{n}\Delta_{n+1}^{-1}(1+c_2 \delta_{n-1}).$ 
 Recall $\theta_{n+1}=c\Delta_{n+1}\delta_{n}$, for all $n \geq 1$. Therefore
 $$N_n\theta_{n+1}\leq \Delta_{n}\Delta_{n+1}^{-1}(1+c_2\delta_{n-1})c\Delta_{n+1}\delta_{n}< c\Delta_{n}\delta_{n}(1+c_2\delta_{n-1}).$$
 We use $\delta_{n} \leq c\Delta_{n}\delta_{n-1} = \theta_{n}$ from \eqref{P3} to get
 $$N_n\theta_{n+1} <c\Delta_{n}\delta_{n}(1+c_2\delta_{n-1})\leq c\Delta_{n}\theta_{n}(1+c_2\delta_{n-1}) \leq c\theta_{n}(1+c_2).$$
 We have $ c(1+c_2)=c(1+4c(c_1+1)).$ We choose $c$ to be small enough such that $c(1+4c(c_1+1)) < 1.$ Thus $N_n \theta_{n+1}< \theta_{n}$ and this completes the proof. \end{proof}
\section{Construction of the $\Gamma$-Besicovitch Set} \label{canti}
\noindent 
We will construct a $1$-dimensional set which contains a translated copy of $e^{-i\theta} \Gamma$, for all $\theta \in [0,1]$. The translated copy of $e^{-i\theta} \Gamma$ will be denoted as  $\Gamma_{\theta}$. A union of finitely many rotated copies of this $1$-dimensional set will contain a copy of $\Gamma$ in every direction.  \par
First we define $\Gamma_{\theta}$, for $\theta \in [0,1]$ that are multiples of $\theta_{n}$ for some $n\in \mathbb{N}$. We will do this inductively. We start by defining $\Gamma_{\theta}$ for multiples of $\theta_1=c$ by  $$ \Gamma_{j \theta_1}=e^{-ij\theta_{1}}\Gamma, \text{ for } 0 \leq j\leq \theta_{1}^{-1}.$$
Assume that for some $n \in \mathbb{N}$, we have defined $\Gamma_{\theta}$, for angles $\theta \leq 1$ that are integer multiples of $\theta_{n}.$ 
We will now define $\Gamma_{\theta}$ for angles that are integer multiples of $\theta_{n+1}$. \par 
We start by defining $\Gamma_{\theta}$ for $ \theta < N_n$. Define 
\begin{equation} \label{q}
    \Gamma_{k\theta_{n+1}}: =e^{-ik\theta_{n+1}}\Gamma+ p_{n,1,k+1}  , \text{ for } 0\leq k < N_n.
\end{equation}
The above definition is motivated by \eqref{lam}. Since the point $p_{n,j,k+l} \in \Gamma$ is the same as $e^{-ik\theta_{n+1}}p_{n,j,l}+p_{n,1,k+1} \in \Gamma_{k\theta_{n+1}}$, there will be a lot of intersection between the $\theta_{n+1}$ neighborhoods of $\Gamma$ and $ \Gamma_{k\theta_{n+1}}$, for $k < N_{n}$. We will show that the Lebesgue measure of the $\theta_{n+1}$ neighborhood of the set $\bigcup_{k < N_{n}}\Gamma_{k\theta_{n+1}}$ is comparable to the Lebesgue measure of the $\theta_{n+1}$ neighborhood of $\Gamma$ which is around $\theta_{n+1}$. This will be a key property to show that the $\Gamma$-Besicovitch set we will construct indeed has dimension $1$.\par   
Now let $\theta =K\theta_{n+1}<\theta_{n}$. We have from Lemma \eqref{Boow} that $N_{n} <\theta_{n} \theta_{n+1}^{-1}$. Thus for some non-negative integer $q$ and for $k < {N_n}$, we have $K=(qN_n+k)$. Define 
\begin{equation} \label{qa}
    \Gamma_{K\theta_{n+1}}: =e^{-iqN_n\theta_{n+1}} \Gamma_{k\theta_{n+1}}.
\end{equation} We then have that
\begin{equation} 
 \bigcup\limits_{K < \theta_{n}\theta_{n+1}^{-1}}^{} \Gamma_{K\theta_{n+1}}  \subseteq  \bigcup_{qN_n<\theta_{n}\theta_{n+1}^{-1}}e^{-iqN_n\theta_{n+1}}\big(\bigcup_{k<N_n}\Gamma_{k\theta_{n+1}}\big). 
\end{equation}
By the induction hypothesis, we have already defined $\Gamma_{j\theta_{n}}$. Let the point $v_{j\theta}\in \mathbb{R}^2$ be such that 
\begin{equation} \label{yal}
    \Gamma_{j\theta_{n}} = e^{-ij\theta_{n}}\Gamma+v_{j\theta_{n}}.
\end{equation}
To define $\Gamma_{\theta}$ for mutiples of $\theta_{n+1}$ in $ [j\theta_n, (j+1)\theta_n)$, we rotate the set $\bigcup_{K < \theta_{n}\theta_{n+1}^{-1}}\Gamma_{K\theta_{n+1}}$ by $-j\theta_{n}$ and place this rotated set at $v_{j\theta_{n}}$.
\begin{equation} \label{akh}
    \bigcup_{K < \theta_{n}\theta_{n+1}^{-1}}\Gamma_{j\theta_{n}+K\theta_{n+1}}=e^{-ij\theta_n} \big( \bigcup_{K < \theta_{n}\theta_{n+1}^{-1}}\Gamma_{K\theta_{n+1}}\big)+v_{j{\theta_n}}.
\end{equation}
Thus we have
\begin{align} \label{qo}
    \Gamma_{j\theta_{n}+K\theta_{n+1}} :=  e^{-ij\theta_{n}} \Gamma_{K\theta_{n+1}} + v_{j\theta_{n}}, \text{ for all } K < \theta_n\theta_{n+1}^{-1}.
\end{align} 
\noindent Observe that if $K=0$, then the above definition agrees with \eqref{yal}. \par 
Recall from \eqref{pe} that $\theta_{n}\theta_{n+1}^{-1} \in \mathbb{N}$. Thus we have defined $\Gamma_{\theta}$ for all $\theta$ which are multiples of $\theta_{n+1}$.  Denote the points $v_{\theta}$ such that $\Gamma_{\theta}:=e^{-i\theta}\Gamma+v_{\theta}$. From \eqref{q}, \eqref{qa} and \eqref{qo} we have 
\begin{equation}\label{akhoo}
    v_{\theta}=
\begin{cases}
p_{n,1,k+1},&  \text{ for } \theta=k\theta_{n+1} \text{ and } k < N_{n}\\
e^{-iqN_{n}\theta_{n+1}} v_{k\theta_{n+1}}, & \text{ for } \theta=(qN_n+k)\theta_{n+1} <\theta_n\\
e^{-ij\theta_{n}}v_{\theta-j\theta_n}+v_{j\theta_{n}}, & \text{ for } \theta \in [j\theta_n, (j+1)\theta_n) \text{ and } \theta \theta_{n+1}^{-1} \in \mathbb{N}.
\end{cases}
\end{equation}
We continue the inductive process for all $n \in \mathbb{N}$. This defines $\Gamma_{\theta}$ for all $\theta$ that are multiples of $\theta_n$,  $n \in \mathbb{N}$. \par  
 We will now introduce rectangles $\mathcal{Q}_{n,j}$  centered at $p_{n,j}$ that contain $Q_{n,j}$. Let
 \begin{align} \label{ne}
     \mathcal{Q}_{n,j}=p_{n,j}+[-C\theta_{n},C\theta_{n}]\times [-C\Delta_n,C\Delta_n],
 \end{align}
 where $C>1$ is a large enough absolute constant to be specified later. We also let
 \begin{align} \label{new}
    T_{n}=\bigcup_{j} \mathcal{Q}_{n,j} \text{ and } T_{n,l\theta_{n}}=e^{-il\theta_n}T_{n}+v_{l\theta_n}.
 \end{align}
 We define
 \begin{equation} \label{Zwe}
     B = \bigcap_{n=1}^{\infty} \bigcup_{l \leq \theta_{n}^{-1}}T_{n,l\theta_{n}}.
 \end{equation}
We claim that $B$ is a $1$-dimensional $\Gamma$-Besicovitch set. \par
For a set $U \subset \mathbb{R}^2$, let $U(\delta)$ denote the open $\delta$-neighborhood of $U$. 
\begin{lemma} The set $B$ defined in \eqref{Zwe} is a 1-dimenionsal set.
    \end{lemma}
    \begin{proof} It is clear from definition \eqref{Zwe} that $B \subset \bigcup_{l \leq \theta_{n+1}^{-1}}T_{n+1,l\theta_{n+1}}.$ We will estimate the measure of the set $\bigcup_{l \leq \theta_{n+1}^{-1}}T_{n+1,l\theta_{n+1}}(\theta_{n+1}). $ \par
 Let $\mathcal{Q'}_{n,j,k}$ be rectangles centered at $p_{n,j,k}$ with size $4C\theta_{n+1} \times 4C \Delta_{n+1}$ i.e.
 $$\mathcal{Q'}_{n,j,k}= p_{n,j,k}+[-2C\theta_{n+1},2C\theta_{n+1}]\times [-2C\Delta_{n+1},2C\Delta_{n+1}]. $$
 From definitions \eqref{ne} and \eqref{new}, we see that $T_{n+1}$ is a union of rectangles centered at points $p_{n,j,k}$ with size $2C\theta_{n+1} \times 2C \Delta_{n+1}$. Therefore
 \begin{equation*} 
  \displaystyle  T_{n+1}(\theta_{n+1})\subset \bigcup_{j}\bigcup_{k \leq N_{n}}\mathcal{Q'}_{n,j,k}.
 \end{equation*}
We let $$T'_{n+1}:=\bigcup_{j}\bigcup_{k \leq N_{n}}\mathcal{Q'}_{n,j,k} \,  \, \text{ and } \, \,  T'_{n+1,l\theta_{n+1}}:=e^{-il\theta_{n+1}}T'_{n+1}+v_{l\theta_{n+1}} .$$ Therefore 
\begin{equation} \label{Any}
  \displaystyle B(\theta_{n+1}) \subset \bigcup_{l \leq \theta_{n+1}^{-1}} T_{n+1,l\theta_{n+1}}(\theta_{n+1})\subset \bigcup_{l \leq \theta_{n+1}^{-1}} T'_{n+1,l\theta_{n+1}}.
 \end{equation}
 From \eqref{lam} we have 
\begin{align*}
    p_{n,j,k+l}=e^{-il\theta_{n+1}}p_{n,j,k}+p_{n,1,l+1}=e^{-i(l-1)\theta_{n+1}}p_{n,j,k+1}+p_{n,1,l}
\end{align*}
and $v_{l\theta_{n+1}}=p_{n,1,l+1}$ from definition \eqref{q}, for $l<N_{n}$. Thus for $l,\, k<N_n$, the center of the two rectangles $e^{-il\theta_{n+1}}\mathcal{Q'}_{n,j,k}+v_{l\theta_{n+1}}$ and $e^{-i(l-1)\theta_{n+1}}\mathcal{Q'}_{n,j,k+1}+v_{(l-1)\theta_{n+1}}$ is $p_{n,j,k+l}$ and the angle between them is $\theta_{n+1}$. For any such pair we have 
\begin{align*}\big|\big(e^{-i(l-1)\theta_{n+1}}\mathcal{Q'}_{n,j,k+1}+v_{(l-1)\theta_{n+1}}\big) \setminus \big(e^{-il\theta_{n+1}}\mathcal{Q'}_{n,j,k}+v_{l\theta_{n+1}}\big) \big| & \lesssim \theta_{n+1} (\max (\theta_{n+1},\Delta_{n+1}))^2 \\
& \leq  \Delta_{n+1}^2\theta_{n+1}.
\end{align*}
Thus we have from Lemma \ref{ba} (ii) that
\begin{align*}
   |T'_{n+1,(l-1)\theta_{n+1}} \setminus T'_{n+1,l\theta_{n+1}}| 
     \lesssim & \, \Delta_{n+1}^2 \theta_{n+1}(\#\{(j,k): Q_{n,j,k} \subset \mathcal{S}_{n+1}, \, k < N_n\}\\
   &  \hspace*{2.46 cm}+\Delta_{n+1} \theta_{n+1}(\#\{j: Q_{n,j, N_n} \subset \mathcal{S}_{n+1}\})  \\
     \lesssim & \, \Delta_{n+1}^2 \theta_{n+1}\#\{(j,k): Q_{n,j,k} \subset \mathcal{S}_{n+1}, k <N_n\} \\ 
     &  \hspace*{3.046 cm} + \Delta_{n+1} \theta_{n+1}(\#\{j: Q_{n,j} \subset \mathcal{S}_{n}\})  \\
     \lesssim & \,  \Delta_{n+1}\Delta_{n}^{-1}\theta_{n+1}.
\end{align*}
\noindent Summing over $l<N_{n}$ and using Lemma \ref{ba} again we get
\begin{align} \label{P}
  \Big|\bigcup_{l<N_{n}} T'_{n+1,l\theta_{n}}\Big| & \lesssim  |T'_{n+1}|+ \Delta_{n+1} \Delta_{n}^{-1}\theta_{n+1}N_{n} \nonumber \\
  & \lesssim \Delta_{n+1} \theta_{n+1}(\#\{(j,k): Q_{n,j,k} \subset \mathcal{S}_{n+1}\})+ \Delta_{n+1} \Delta_{n}^{-1}\theta_{n+1}N_{n} \nonumber \\
  & \lesssim \theta_{n+1}. 
\end{align}
Let $K <\theta_{n} \theta_{n+1}^{-1}$ and let $K=qN_n+l$ and $l <N_{n}.$ From \eqref{akhoo}, we have 
\begin{align*}
   T'_{n+1,K\theta_{n+1}}& =e^{-iK\theta_{n+1}}T'_{n+1}+v_{K\theta_{n+1}}  \\
   & = e^{-iqN_n \theta_{n+1}} e^{-il\theta_{n+1}} T'_{n+1}+e^{-iqN_n \theta_{n+1}} v_{l\theta_{n+1}} \\
& = e^{-iqN_n \theta_{n+1}} (e^{-il\theta_{n+1}} T'_{n+1}+v_{l\theta_{n+1}}) =e^{-iqN_n \theta_{n+1}} T'_{n+1,l\theta_{n+1}}.
\end{align*} Therefore 
\begin{align*}
 \bigcup_{K< \theta_{n}\theta_{n+1}^{-1}}T'_{n+1,K\theta_{n+1}}   \subset \bigcup_{qN_n\leq \theta_{n}\theta_{n+1}^{-1}} \bigcup_{l < N_n} T'_{n+1,l\theta_{n+1}}.
\end{align*}
Thus using the estimate \eqref{P}, we obtain
\begin{align} \label{Nip}
    \Big|\bigcup_{K < \theta_{n}\theta_{n+1}^{-1}}T'_{n+1,l\theta_{n+1}}\Big| \lesssim \theta_{n+1} N_n^{-1}\theta_{n}\theta_{n+1}^{-1}\leq  N_n^{-1}\theta_{n}.
\end{align} We will now sum over all $\theta \leq 1$ which are multiples of $\theta_{n+1}$. Let $\theta=j\theta_{n}+ K \theta_{n+1}$ such that $K\theta_{n+1} < \theta_n$. From \eqref{akhoo} we have
\begin{align*}
   T'_{n+1,\theta_{}}& =e^{-i\theta}T'_{n+1}+v_{\theta}  \\
   & = e^{-ij \theta_{n}} e^{-i(\theta-j\theta_{n})} T'_{n+1}+e^{-ij \theta_{n}}v_{\theta-j\theta_{n}} +v_{j\theta_{n}}. \\
& = e^{-ij \theta_{n}} (e^{-iK\theta_{n+1}} T'_{n+1}+v_{K\theta_{n+1}})+ v_{j\theta_{n}}=e^{-ij\theta_n } T'_{n+1,K\theta_{n+1}}+v_{j\theta_{n}}.
\end{align*} 
Therefore from Lemma \ref{ba} (i) we have that
\begin{align} \label{N}
    \Big|\bigcup_{l\leq \theta_{n+1}^{-1}}T'_{n+1,l\theta_{n+1}}\Big| \leq \Big|\bigcup_{l< \theta_{n}\theta_{n+1}^{-1}}T'_{n+1,l\theta_{n+1}}\Big|\theta_{n}^{-1}\lesssim N_n^{-1}\theta_{n}\theta_{n}^{-1}\lesssim \Delta_{n+1}\Delta_{n}^{-1}.
\end{align}
Thus from \eqref{Any} and \eqref{N} we have that 
$$ |B(\theta_{n+1})|\lesssim\Delta_{n+1}\Delta_{n}^{-1}.$$
From \eqref{P2}, we see that $\Delta_{n+1}^{\frac{1}{n+1}} \leq c^{\frac{1}{n+1}} \delta_{n}\leq  \delta_{n} \leq \Delta_n$. Thus $ B(c\Delta_{n+1}^{1+\frac{1}{n+1}}) \subset B(c\Delta_{n+1}\delta_{n}) =B(\theta_{n+1}) $ and we have
$$|B(c\Delta_{n+1}^{1+\frac{1}{n+1}})|   \lesssim \Delta_{n+1}\Delta_{n}^{-1} \lesssim \Delta_{n+1}^{1-\frac{1}{n+1}} \leq \big(\Delta_{n+1}^{1+\frac{1}{n+1}}\big)^{1-\frac{2}{n+1}}. $$
Therefore for all $m>n$
\begin{align*}
  |B(c\Delta_{m+1}^{1+\frac{1}{m+1}})| \lesssim    \big(\Delta_{m+1}^{1+\frac{1}{m+1}}\big)^{1-\frac{2}{m+1}} \lesssim \big(\Delta_{m+1}^{1+\frac{1}{m+1}}\big)^{1-\frac{2}{n+1}}.
\end{align*} Since $\Delta_{m+1} \rightarrow 0$ as $m \rightarrow \infty$ we conclude that 
$$ \dim_{H}B \leq 2-(1-\frac{2}{n+1})=1+\frac{2}{n+1}.$$
Taking $n \rightarrow \infty, $ we see that $\dim_{H} B=1$.
\end{proof}
It remains to prove that $B$ contains a rotated copy of $\Gamma$ in every direction. \\
Let $D=\{j\theta_n : \text{for some }n \in \mathbb{N} \text{ and }j\leq \theta_{n}^{-1}\}$, the dense subset of $[0,1]$ for which we have already defined $v_{\theta}$. We will now define $v_{\theta}$ for all $\theta \in [0,1]$ as the following left limit. 
\begin{equation} \label{zooni} 
v_\theta=\lim_{\substack{ \hspace*{.25 cm} \theta' \to \theta^{-} \\ \theta' \in D }} v_{\theta'}.
\end{equation}
It will follow from Lemma \ref{chu} below that the above limit exists. We define $\Gamma_{\theta}$ for all $\theta \in [0,1]$ as before 
$$\Gamma_{\theta}=e^{-i\theta}\Gamma+v_{\theta}.$$ 
Therefore we will have the following.
    \begin{equation} \label{zoni}
    \Gamma_{\theta}=e^{-ij\theta_{n}}(\Gamma_{\theta-j\theta_{n}})+v_{j\theta_{n}}, \text{ for all } \theta \in [j\theta_{n},(j+1)\theta_{n}).
\end{equation}
\begin{lemma} \label{chu} The limit in \eqref{zooni} exists. Moreover for $\theta < \theta _{n}$, we have 
\begin{equation} \label{eq}
|P_{x}(v_{\theta})| \lesssim \theta_{n} \text{ and } |P_{y}(v_{\theta})| \lesssim \Delta_{n}, 
\end{equation}
where the implicit constant is absolute and independent of $n$.
\end{lemma} 
\begin{proof}
If $\theta=k\theta_{n+1} \text{ and } k < N_{n}$, then we have from \eqref{q} that $v_{\theta}=p_{n,1,k+1}$. By definition of $N_n$ (see \eqref{h}) this implies that $v_{\theta}$ lies in the rectangle $Q_{n,1}=[0, \delta_n]\times [0,\Delta_n]$.  Therefore $|v_{\theta}|\leq 2\Delta_{n}.$  \par 
If $\theta =K\theta_{n+1}<\theta_{n} $ and $K=qN_n+k$ with $k < N_n$, then we have from \eqref{qa} that 
$$ v_{\theta}=e^{-iqN_n\theta_{n+1}}(v_{k\theta_{n+1}}).$$ 
We are rotating a vector of length at most $2\Delta_{n}$ around origin by angle of magnitude at most $\theta_{n}$. Therefore we get
$$\|v_{\theta}-v_{k\theta_{n+1}}\| =\|v_{k\theta_{n+1}}\|\theta \leq2\Delta_{n} \theta_{n}$$ and 
\begin{equation}\label{cu}
    |P_{x}(v_{\theta})| \leq |P_{x}(v_{k\theta_{n+1}})| +  2\Delta_{n}\theta_{n} \leq \delta_{n}+ 2\Delta_{n}\theta_{n}
    \end{equation} and 
    \begin{equation} \label{cv}
    |P_{y}(v_{\theta})| \leq |P_{y}(v_{k\theta_{n+1}})|+  2\Delta_{n}\theta_{n} \leq \Delta_{n}+2\Delta_{n}\theta_{n}. \end{equation}
From \eqref{akhoo}, we have that if $\theta \in [j\theta_{n}, (j+1)\theta_{n})$ and is a multiple of $\theta_{n+1}$, then 
    $$v_{\theta}-v_{j\theta_n} =e^{-ij\theta_{n}}(v_{\theta-j\theta_{n}}).$$
We are rotating a vector of length at most $2\Delta_{n}$ around origin by angle of magnitude at most $1$. Therefore we get
\begin{equation} \label{tyuu}
    \|v_{\theta}-v_{j\theta_n}\|=\|v_{\theta-j\theta_{n}}\|j\theta_{n} \leq 2\Delta_{n}, \text{ for all } \theta \in [j\theta_{n}, (j+1)\theta_{n}) \text{ and } \theta \theta_{n+1}^{-1} \in \mathbb{N}.
\end{equation}
Now let $\theta \in [0,1]$. Let ${\Theta}_{m}=\lfloor \theta\theta_{m}^{-1} \rfloor \theta_{m}, $ where $\lfloor\cdot \rfloor$ denotes the greatest integer function. Then $\Theta_{m} $ denotes the greatest integer multiple of $\theta_m$ less than $\theta$. From \eqref{tyuu} we have 
$$\|v_{\Theta_{m+1}}-v_{\Theta_m}\|\leq 2\Delta_{m}, \text{ for all } m \in \mathbb{N}. $$ Therefore
\begin{align} \label{cho}
   \big\|v_{\Theta_{m+l}}-v_{\Theta_m} \big\|=\big\|\sum_{j=m}^{m+l-1}(v_{\Theta_{j+1}}-v_{\Theta_j})\big\|
     \leq \sum_{j=m}^{m+l-1}\big\|v_{\Theta_{j+1}}-v_{\Theta_j} \big\|\leq \sum_{j=m}^{m+l-1} 2 \Delta_{j}
     \lesssim \Delta_{m}.
\end{align} 
We have used that $\Delta_{m+1}\leq c\delta_{m}^{m+1} \leq \Delta_{m}^{m+1}$ which follows from \eqref{P2}. As $\Delta_{m} \rightarrow 0$, we conclude that 
$\{v_{\Theta_{m}}\}_{m \in \mathbb{N}}$ converges. Thus the limit $v_{\theta}$ as defined in \eqref{zooni} exists. \par
Now we again let $\theta < \theta_{n}$. We have from \eqref{cho} that $\|v_{\Theta_{n+1+m}}-v_{\Theta_{n+1}}\| \lesssim \Delta_{n+1}.$ Therefore we have
\begin{equation*}
    \|v_\theta-v_{\Theta_{n+1}}\| \lesssim \Delta_{n+1}.
\end{equation*}
Combining the above equation with \eqref{cu} and \eqref{cv}, we get for $\theta <\theta_{n}$
\begin{equation*} 
    |P_x(v_\theta)| \leq  |P_x(v_{\Theta_{n+1}})| +\|v_\theta-v_{\Theta_{n+1}}\| \lesssim \delta_n+2\Delta_n\theta_{n}+\Delta_{n+1} \lesssim \Delta_{n}\delta_{n-1}=\theta_{n}.
\end{equation*} and 
\begin{equation*} 
    |P_y(v_\theta)| \leq  |P_y(v_{ \Theta_{n+1}})| +\|v_\theta-v_{\Theta_{n+1}}\| \lesssim \Delta_n+2\Delta_n\theta_{n}+\Delta_{n+1} \lesssim \Delta_{n}.
\end{equation*} \end{proof}
We are now ready to prove that $B$ is a $\Gamma$-Besicovitch set.
\begin{lemma} \label{lemm}For  all $n \in \mathbb{N}$ and $\theta \in [0,1]$, we have 
\begin{equation} \label{lem}
    \Gamma_{\theta} \subset \bigcup_{i \leq \theta_{n}^{-1}} T_{n,i\theta_{n}}.
\end{equation}
\end{lemma} 
\begin{proof} For any $\theta < \theta_n$ and $p_{n,j} \in \Gamma \subset [0,1]^2$, we have
$$\|e^{-i\theta}p_{n,j}-p_{n,j}\| \leq \|p_{n,j}\|\theta \lesssim \theta_n.$$
 The rectangle $Q_{n,j}$ has sides parallel to the coordinate axes. Therefore, $$|P_x(e^{-i\theta}Q_{n,j})| \lesssim |P_{x}(Q_{n,j})|+\Delta_n\theta \lesssim \delta_n+\Delta_{n}.\theta_n \lesssim \theta_n $$ and $$|P_y(e^{-i\theta}Q_{n,j})| \leq |P_{y}(Q_{n,j})|+\Delta_n\theta \leq 2\Delta_{n}+2\delta_n\theta_n \lesssim \Delta_{n}.$$
 Let $z \in \Gamma_{\theta}$. Since $\Gamma_{\theta} \subset e^{-i\theta}(\bigcup\hspace{.03 cm}Q_{n,j}) + v_{\theta}$, the point $z -v_{\theta}$ lies in the rectangle $e^{-i\theta}Q_{n,j}$ for some $j$. Thus from the above discussion and Lemma \ref{chu}, we get 
\begin{equation*}
    |P_x(z-p_{n,j})| \leq |P_x(z-v_{\theta}-e^{-i\theta}(p_{n,j}))| +|P_x(e^{-i\theta}(p_{n,j})-p_{n,j})|+|P_x(v_{\theta})|\lesssim \theta_n. \end{equation*}
    and
  \begin{equation*}
    |P_y(z-p_{n,j})| \leq |P_y(z-v_{\theta}-e^{-i\theta}(p_{n,j}))| +|P_y(e^{-i\theta}(p_{n,j})-p_{n,j})|+|P_y(v_{\theta})|\lesssim \Delta_n. \end{equation*}
We take $C$ in \eqref{ne} large enough such that $\Gamma_{\theta} \subset \bigcup \mathcal{Q}_{n,j}=T_{n}$, for all $\theta < \theta_n $. Now let $\theta \in [j\theta_{n},(j+1)\theta_{n}))$ then by \eqref{zoni}, we obtain 
\begin{equation}
  \Gamma_{\theta}=e^{-ij\theta_{n}}\Gamma_{\theta-j\theta_{n}}+v_{j\theta_n} \subset e^{-ij\theta_{n}}T_{n}+v_{j\theta_n}= T_{n,j\theta_{n}}.
\end{equation} This completes the proof.
\end{proof}
\noindent Thus $\Gamma_{\theta} \subset B$ for every $\theta \in [0,1]$ and so we have proved Theorem \ref{main} (ii).

\section{Dimension of the Domain of $\Gamma$}
So far we have proved that if $\{\delta_{n}\}_{n \in \mathbb{N}}$ and $\{\Delta_{n}\}_{n \in \mathbb{N}}$ satisfy \eqref{P1}-\eqref{P4}, and if $\Gamma$ and $B$ are constructed as in Section \ref{cavt} and Section \ref{canti} respectively, then $\Gamma$ has property (i), (ii) and (iii) of Theorem \ref{main}. It remains to prove that the dimension of the $x$-projection of $\Gamma$ can take any value $s \in [0,1]$. \par 
 Let $s \in [0,1]$ and let $\{\delta_{n}\}_{n \in \mathbb{N}}$ and $\{\Delta_{n}\}_{n \in \mathbb{N}}$ satisfy conditions \eqref{P1}, \eqref{P2}, \eqref{P4} and 
\begin{align} \label{n}
   \delta_{n} = c\Delta_{n}^{\frac{1}{s_{n}}},
\end{align}
where $$s_{n}= \begin{cases} 
      \frac{1}{n}, & \text{ if } s=0 \\
      \frac{sn}{n+1}, &  \text{ if }0 < s \leq 1.
      \end{cases}$$
   We will show that \eqref{P3} holds and the corresponding $\Gamma$ we constructed in Section \ref{cavt} satisfies
\begin{align} \label{na}
    \dim (P_{x}(\Gamma))=s.
\end{align}
     From \eqref{P2}\,($\Delta_{n+1}^{\frac{1}{{n+1}}}\leq c^{\frac{1}{{n+1}}}\delta_{n}$) and  \eqref{n} we have
$$ \delta_{n+1}=c\Delta_{n+1}^{\frac{1}{s_{n+1}}} =c\Delta_{n+1} \Delta_{n+1}^{\frac{1}{s_{n+1}}-1}  \leq c\Delta_{n+1} (c\delta_{n}^{n+1})^{\frac{1}{s_{n+1}}-1}.$$
If $s=0$ then $$(n+1)(s_{n+1}^{-1}-1)=(n+1)n \geq 1.$$ If $s>0$ then $$(n+1)(s_{n+1}^{-1}-1)=(n+1)(s^{-1}(n+2)/(n+1)-1) \geq 1.$$ Hence in both cases $$\delta_{n+1} \leq c\Delta_{n+1}(c\delta_{n}^{n+1})^{\frac{1}{s_{n+1}}-1} \leq c^{1+\frac{1}{n+1}}\Delta_{n+1}\delta_{n} \leq c\Delta_{n+1}\delta_{n}.$$ So indeed \eqref{P3} holds. 
\begin{proof}[Proof of \eqref{na}]
We first show that $\dim (P_{x}(\Gamma)) \leq s$. By construction, for all $p\in \mathbb{N}$, $\bigcup_{j}P_{x}(Q_{p,j})$ is a cover of $P_{x}(\Gamma)$.  \par 
Assume $s=0$ and let $s' >0$. Using Lemma \ref{ba} (ii) and \eqref{n} we have
\begin{align*}
    \sum_{j} |P_{x}(Q_{p,j})|^{s'} = \delta_{p}^{s'}\#\{j: Q_{p,j}\subset \mathcal{S}_p\} \lesssim \delta_{p}^{s'} \Delta_{p}^{-1} \lesssim \delta_{p}^{s'-\frac{1}{p}}.
\end{align*} Hence for any $p \in \mathbb{N}$ such that $1/p< s'$, we have 
$$ \sum_{j} |P_{x}(Q_{p,j})|^{s'} \lesssim \delta_{p}^{s'-\frac{1}{p}} <1.$$
As $p \rightarrow \infty$, $\delta_{p} \rightarrow 0$. Therefore $\dim (P_x(\Gamma)) \leq s'$. Since $0<s'$ was arbitrary we conclude that $\dim (P_x(\Gamma)) =0.$ \par
Now assume that $s>0$. We again use Lemma \ref{ba} (ii) and \eqref{n} to get \begin{align*}
    \sum_{j} |P_{x}(Q_{p,j})|^{s}= \delta_{p}^{s}\#\{j: Q_{p,j} \subset \mathcal{S}_p\} \lesssim \delta_{p}^{s} \Delta_{p}^{-1} \lesssim \delta_{p}^{s}\delta_{p}^{-s_{p}}= \delta_{p}^{\frac{s}{p+1}}\leq 1.
\end{align*} Therefore $\dim (P_{x}(\Gamma)) \leq s$. \par
It remains to prove the opposite inequality for $s>0$. Fix $m \in \mathbb{N}$, and let $\bigcup_{i}I_i$ be a cover of $P_x(\Gamma)$ by open intervals $I_i$ such that $|I_i| < \delta_{m}$, for all $i$. We can assume it is a finite cover since $P_x(\Gamma)$ is compact. Let $p \in \mathbb{N}$ be such that $\delta_{p}< \min\limits_{i} |I_i|$.\par 
Now fix an arbitrary $i$ and let $n \in \mathbb{N}$ be such that $\delta_{n+1} \leq  |I_{i}| <  \delta_n$. We will show that 
\begin{align} \label{esta}
 |I_i|^{s_m} & \gtrsim \sum\limits_{ \{(j,k):\,P_{x}(Q_{n,j,k})\cap I_i \neq \emptyset \}} |P_{x}(Q_{n,j,k})|^{s_{n+1}}.
\end{align}
Recall from \eqref{bux} that
$ 3\delta_{n} \leq P_{x}(p_{n,j+1})- P_{x}(p_{n,j})$. Since the distance between the left endpoints of $P_x(Q_{n,j})$ and $P_x(Q_{n,j+1})$ is greater than $3\delta_{n}=3|P_x(Q_{n,j})|$, we can assume that $I_i$ intersects  $P_x(Q_{n,j})$ for only one rectangle $Q_{n,j}$. Let 
\begin{equation}
    \lambda = \displaystyle {\frac{|I_i|}{\delta_n \Delta_{n}^{-1}\Delta_{n+1}}} .
\end{equation} We see from \eqref{bux} that $P_x(p_{n,j,k+1})-P_x(p_{n,j,k}) \gtrsim \delta_{n} \Delta_{n}^{-1}\Delta_{n+1}$. Therefore we have
\begin{align*}
    \#\{Q_{n,j,k} : P_{x}(Q_{n,j,k})\cap I_i \neq \emptyset \} \lesssim \lambda+1.
\end{align*} We look at two cases. \\
\textbf{Case 1}: $\lambda < 1$. \par 
\hspace*{.34 cm} In this case $I_i$ intersects $P_{x}(Q_{n,j,k})$ for at most constant many rectangles $Q_{n,j,k}$. Since $\delta_{n+1} \leq |I_i| < \delta_m$ we have $m < n+1$. Therefore $s_{m} <s_{n+1}$ and 
$$|I_i|^{s_m} > |I_i|^{s_{n+1}} \geq \delta_{n+1}^{s_{n+1}} \gtrsim \sum\limits_{ \{(j,k):\,P_{x}(Q_{n,j,k})\cap I_i \neq \emptyset \}} |P_{x}(Q_{n,j,k})|^{s_{n+1}}. $$
So \eqref{esta} holds. \\
\textbf{Case 2}: $\lambda \geq 1$. \par 
To conclude \eqref{esta} we need to verify that 
\begin{align*}
   |I_i|^{s_{m}}=(\lambda \delta_n \Delta_{n}^{-1}\Delta_{n+1})^{s_{m}} & \gtrsim (\lambda+1) \delta_{n+1}^{s_{n+1}} \simeq \lambda \delta_{n+1}^{s_{n+1}}.
\end{align*} Rewriting the above inequality we get
\begin{align} \label{Tes}
    ( \delta_{n}\Delta_{n}^{-1}\Delta_{n+1})^{s_m} \gtrsim \lambda^{1-s_m} \delta_{n+1}^{s_{n+1}}.
\end{align} Using $s_m \leq s_{n}$ and \eqref{n} we have that
$$\delta_{n}^{s_m}\geq \delta_{n}^{s_n}\simeq \Delta_{n} \simeq \Delta_{n}\Delta_{n+1}^{-1}\delta_{n+1}^{s_{n+1}}.$$
Thus 
$$ (\delta_{n}\Delta_{n}^{-1}\Delta_{n+1})^{s_m} \gtrsim (\Delta_{n}\Delta_{n+1}^{-1})^{1-s_m}\delta_{n+1}^{s_{n+1}}. $$ 
Since $|I_i| < \delta_n$ we have $\lambda < \Delta_{n} \Delta_{n+1}^{-1}$. As $s_m <1$, we have $ \lambda^{1-s_{m}} < (\Delta_{n} \Delta_{n+1}^{-1})^{1-s_{m}}$. Therefore we obtain \eqref{Tes} from the above inequality. Thus \eqref{esta} holds in Case 2. \par
Since $\delta_{p}<|I_i| < \delta_{n}$ we have that $p \geq  n+1$. By the definition \eqref{h} of $N_{n}$ and Lemma 3 (i) we have
\begin{align*}
    \#\{Q_{p,l}: Q_{p,l} \subset Q_{n,j,k}\} \leq \prod_{q=n+1}^{p-1}\Delta_{q}\Delta_{q+1}^{-1}(1+c_2\delta_{q-1})=\Delta_{n+1}\Delta_{p}^{-1}\prod_{q=n+1}^{p-1}(1+c_2\delta_{q-1}).
\end{align*} 
As noted earlier below \eqref{Ramen}, the product $\prod_{q=n+1}^{p}(1+c_2\delta_{q-1}) $ converges to an absolute constant. Therefore $\prod_{q=n+1}^{p}(1+c_2\delta_{q-1}) \simeq 1$ and 
\begin{align} \label{establ} 
\sum\limits_{\{l : Q_{p,l}\subset Q_{n,j,k}\}} |P_{x}(Q_{p,l})|^{s_{p}} \lesssim \delta_{p}^{s_{p}}\Delta_{n+1}\Delta_{p}^{-1} \simeq \Delta_{n+1} \simeq \delta_{n+1}^{s_{n+1}} =
 |P_{x}(Q_{n,j,k})|^{s_{n+1}}.
\end{align} 
 Hence we obtain from \eqref{establ} and \eqref{esta} that 
 \begin{align*}
   |I_{i}|^{s_{m}} \gtrsim \sum\limits_{ \{(j,k):\,P_{x}(Q_{n,j,k})\cap I_i \neq \emptyset \}} |P_{x}(Q_{n,j,k})|^{s_{n+1}} & \gtrsim \sum\limits_{ \{(j,k):\,P_{x}(Q_{n,j,k})\cap I_i \neq \emptyset \}} \sum\limits_{ \{l:\,Q_{p,l} \subset Q_{n,j,k}\}}  |P_{x}(Q_{p,l})|^{s_{p}} \\
   & \geq \sum\limits_{ \{l:\,P_{x}(Q_{p,l})\cap E_i \neq \emptyset \}} |P_{x}(Q_{p,l})|^{s_{p}}.
 \end{align*} 
 Summing over all $i$ we get 
\begin{align}
    \sum_{i}|I_i|^{s_{m}} \gtrsim \sum_{i} \sum\limits_{ \{l:\,P_{x}(Q_{p,l})\cap I_i \neq \emptyset \}}  |P_{x}(Q_{p,l})|^{s_{p}} \geq  \sum_{l } |P_{x}(Q_{p,l})|^{s_{p}} \gtrsim \delta_{p}^{s_{p}} \Delta_{p}^{-1} \simeq 1.
\end{align} This proves $\dim (P_{x}(\Gamma)) \geq s_m$. Taking $m \rightarrow \infty$ we get \eqref{na}.
\end{proof} 
\section{Acknowledgement}
I would like to thank Marianna Cs\"ornyei for asking this question, for her unparalleled guidance and patience, especially when I was writing this paper. I would also like to thank Korn\'elia H\'era for helpful discussions and suggestions.
\bibliographystyle{ieeetr}
\bibliography{references}
\nocite{CHL}
\nocite{Fa}
\nocite{HKM}
\nocite{Ma95}
\nocite{HL}
\end{document}